\documentclass{aptpub}

\usepackage[colorlinks = true,linkcolor = blue,urlcolor  = blue,citecolor = blue,anchorcolor = blue]{hyperref}
\authornames{Horesh {\it et al}} 
\shorttitle{Schatten $p$ Norm} 

\usepackage[utf8]{inputenc}

\usepackage{xcolor}

\usepackage{bbm}

\newcommand{\ostar}{\textcircled{$\star$}}

\newcommand{\pnCycles}{\Sigma^{p,n}_{{\scriptscriptstyle{\uparrow}}}}
\newcommand{\EX}{\mathbb{E}}   
\newcommand{\eV}{{\rm Eval}}
\newcommand{\ostr}{{\rm\ostar}}

\newcommand{\tr}{\mbox{\sf Tr}}

\newcommand{\Real}{\mathbb R}

\newcommand{\Nat}{\mathbb N}

\newcommand{\bS}{{\bf S}}               

\newcommand{\bX}{{\bf X}}               

\newcommand{\bk}{{\bf k}}
\newcommand{\bm}{{\bf m}}

\newcommand{\al}{\alpha}                
\newcommand{\Lam}{\Lambda}               
   
\usepackage[linesnumbered,ruled]{algorithm2e}

\begin{document}

\title{On The Variance of Schatten $p$-Norm Estimation with Gaussian Sketching Matrices}

\authorone[IBM Research]{Lior Horesh} 
\authortwo[IBM Research]{Vasileios Kalantzis}
\authorthree[IBM Research]{Yingdong Lu}
\authorfour[IBM Research]{Tomasz Nowicki}

\addressone{1101 Kitchawan Rd, Yorktown Heights, NY 10598, U.S.A.} 
\emailone{\{
lhoresh,yingdong,tnowicki\}@us.ibm.com and 
\{vkal\}@ibm.com} 

\begin{abstract}
Monte Carlo matrix trace estimation is a popular randomized technique to estimate the trace of implicitly-defined matrices via averaging quadratic forms across several observations of a random vector. The most common approach to analyze the quality 
of such estimators is to consider the variance over the total number of observations. 
In this paper we present a procedure to compute the variance of the estimator proposed 
by Kong and Valiant [\emph{Ann. Statist.} 45 (5), pp. 2218 - 2247] for the case of Gaussian random vectors and provide a sharper bound than previously available.
\end{abstract}



\section{Introduction}
\label{sec:intro}

For an integer $p>1$ and a general $d\times m$ matrix $B$, the Schatten $p$-norm (also known as Schatten-von Neumann $p$-norm) is defined as,
\begin{equation*}
\|B\|_{p} : = \left(\sum\limits_{j=1}^{ d\wedge m} \sigma_{j,B}^p\right)^{1/p}= 
\left(\tr[( BB^\top  )^{p/2}]\right)^{1/p},
\end{equation*}
where $d\wedge m :=\min \{d,m\}$ and $\sigma_{j,B}$ denotes the $j$-th singular value of the matrix $B$.  The Schatten $p$-norm is a generalization of the Frobenius norm ($p=2$) and holds an important role in statistics and optimization, see e.g.~\cite{7539605},~\cite{Martinsson_Tropp_2020}, \cite{kalantzis2024asynchronous},  and~\cite{10.1214/16-AOS1525}. An unbiased estimator of general Schatten $p$-norm was presented in~\cite{10.1214/16-AOS1525} along with a bound on its variance. The rationale  behind this estimator is based on the fact that, for a random \emph{isotropic} vector $w\in \Real^d$ (i.e., $\EX[ww^\top]=I_d$), we have 
\begin{align}\label{eqn:sampletrace}\EX[w^\top(Aw)]=\tr(A).\end{align} 
As an extension, for any $d\times m$ matrix $B$, $\EX[w^\top B^\top B w]=\tr[B^\top B  ]=\|B\|_F^2$. Hence, a sample mean of $w^\top B ^\top B w$ is an unbiased estimator of $\|B\|_F^2$. For general integers $p>1$, the following quantity is an unbiased estimator of $\|B\|^{2p}_{2p}$, for some integer\footnote{A recommended lower bound for the value of $n$ can be computed via exploiting the work in \cite{li2014sketching}.} $n\ge p$, 
\begin{align*}
V^p_n = {\binom{n}{p}}^{-1}
\sum_{1\le i_1< i_2< \ldots <i_p \le n} (W)_{i_1i_2} (W)_{i_2i_3}\ldots (W)_{i_pi_1}\,,
\end{align*}
with $(W)_{ij} =w_i^\top B^\top B w_j$ and $w_i,w_j\in\Real^d$ independent for $i\neq j$.
Further details on estimating the Schatten $p$-norm can be found in \cite{martinsson2020randomized} while a Monte Carlo algorithm can be found in \cite{etna_vol55_pp213-241}. 
The key for assessing the quality of such an estimator is to have an accurate estimation 
of the variance of~$V^p_n$.  In Proposition 4 of~\cite{10.1214/16-AOS1525}, the variance 
of $V^p_n$ exploits vectors $w$ whose marginal fourth moments are bounded by a positive constant $\kappa$, leading to the upper bound
\begin{align}
\label{eqn:KongValiantBound}
2^{12p}p^{6p}\kappa^p\max\left(\frac{d^{p-2}}{n^p}, \frac{d^{\frac12-\frac{1}{p}}}{n}\right)\tr^2((B^\top B)^p).
\end{align}

The main result of this paper is the computation of an improved upper bound in the case that $w$ is Gaussian. To obtain an explicit expression of the variance of $V_n^p$  we consider an $n\times d$ \emph{sketching matrix} $\bX=(X_1, X_2, \ldots, X_n)^\top$  where each row $X^\top_i,i=1,2, \ldots n$, is an independent sample of the $d$-variate normal variable $N(0, I_d)$, and denote $\bS:=B^\top B$, which is a $d\times d$ symmetric matrix. In this case, the estimator takes the form 
\begin{align} \label{eqn:defn_Vp}
    V^p_n = {\binom{n}{p}}^{-1}
\sum_{1\le i_1< i_2< \ldots <i_p \le n} \prod_{\ell=1}^{p-1}(\bX\bS\bX^\top)_{i_\ell i_{\ell+1}}\,.
\end{align}
The variance of the estimator~\eqref{eqn:defn_Vp} is expressed in Theorem~\ref{thm:varrep} using the norms of powers of $\bS$. This result is obtained by transforming the calculation of the terms in $\EX[(V^p_n )^2]$ to a trace form, as presented in Lemma~\ref{lem:twocycles}. In Section~\ref{sec:MNCalc}, we investigate the recursive structure of the trace form, which provides the main technical results of this paper. Theorem~\ref{thm:varrep} allows us to provide a new upper bound of the variance of the estimator in Theorem~\ref{thm:upper_bound}. This upper bound is the sum of four separate terms which are  $\tr^2(\bS^p)$, the square of the expectation multiplied by some factors. The first two terms have factors bounded by one when $n\ge 2p$ and each of these two terms decreases at least of order $O(1/n)$ as $n$ grows. In contrast, when $n<2p$, these two terms are equal to zero. The third term is bounded by $\frac{(3pd)^p}{(dn)^2}\cdot\tr^2(\bS^p)$, while the fourth term is bounded by $\frac{(12pd)^{p/2}}{dn^{p/2}}\cdot\tr^2(\bS^p)$. The sum of the four upper bounds of these terms is $(2+ \frac{(3pd)^p}{(dn)^2}+\frac{(12pd)^{p/2}}{dn^{p/2}})\cdot\tr^2(\bS^p)$, significantly smaller than the bound in~\eqref{eqn:KongValiantBound}. Notice that as out Theorem~\ref{thm:varrep} suggests, the factor $d^{p-2}$ is unavoidable. Indeed, consider the case that all the singular values of $B$ are roughly the same. Then the variance is roughly equal to $d^p$ as it is in the form of Schatten $p$ norms with $p\ge 2$. In this case, $\tr^2(\bS^p)$ is of the order of $d^2$. 

While finalizing this paper, we became aware of~\cite{ChuCortinovis2024}. Therein, the authors also consider the problem of obtaining an improved upper bound of the variance in~\cite{10.1214/16-AOS1525} with Gaussian sketching matrices. The authors in~\cite{ChuCortinovis2024} obtain the bounds described therein through a significantly different approach and these bounds are expressed in terms of Schatten norms of various orders that can be as high as $4p$, unlike the explicit bound presented in Theorem~\ref{thm:upper_bound} of the present paper. Moreover, in 
contrast to~\cite{ChuCortinovis2024} and other related work, Theorem~\ref{thm:varrep} of the this paper 
presents the exact relation between the Schatten norms. 

\section{Preliminaries and Notations}

\begin{defn}
\label{defn:cycle}
Given $n, k\in \Nat$, a $k$-cycle is a sequence of $k$ distinct integers, $\sigma=(\sigma_1,\ldots,\sigma_k)$ with $\sigma_i\in [n]\equiv[1,\ldots,n]$ for each $i=1,2,\ldots, k$. An increasing $k$-cycle $\sigma=(\sigma_1,\ldots\sigma_k)$ is s a $k$-cycle with the additional property that $\sigma_1<\sigma_2<\ldots<\sigma_k$. For an $n\times n$
matrix $A$, each cycle $\sigma$ defines a number:
\begin{align*}
A_\sigma= \prod_{i=1}^k A_{\sigma_i, \sigma_{i+1}}
\end{align*}
with the convention $\sigma_{k+1}=\sigma_1$. 
\end{defn}
Let $\pnCycles$ denote the collection of all increasing $p$-cycles, now we can write
\[
V_p^n={\binom{n}{p}}^{-1}\sum_{\sigma\in \pnCycles}(\bX \bS \bX^\top)_{[\sigma]}\,.
\]
For any two $p$-cycles, $\sigma$ and $\tau$, let $\gamma=(\gamma_1,\gamma_2,\ldots,\gamma_q)$ be the increasing $q$-cycle of all their common elements. Slightly abusing the notation, whenever convenient we shall use the symbol of a cycle to denote the set of its elements, for example $\gamma= \sigma \cap \tau$. Furthermore, define $k^{\sigma, \tau}_i, m^{\sigma,\tau}_i,$ $i=0,1,\dots,q$, by:
\begin{align*}
\begin{array}{llc}
k^{\sigma, \tau}_0:=|\{j: j\in\sigma, j< \gamma_1\}|, 
& m^{\sigma, \tau}_0:=|\{j: j\in\tau, j< \gamma_1\}|,
&i=0
\\ 
k^{\sigma, \tau}_i := |\{j: j\in\sigma, \gamma_i \le j< \gamma_{i+1}\}|, 
& m^{\sigma, \tau}_i := |\{j: j\in\tau, \gamma_i \le j< \gamma_{i+1}\}|,
& 1\le i < q, 
\\
k^{\sigma, \tau}_q:= |\{j: j\in\sigma, j \ge \gamma_q\}|, 
& m^{\sigma, \tau}_q:= |\{j: j\in\tau, j \ge \gamma_q\}|,
&i=q.
\end{array}
\end{align*} 
In the following, we omit the superscripts whenever it does not cause any ambiguity.  
By definition, the cardinality of an empty set is zero. When $q=0$, we set $k_0=m_0=p$.
\begin{example}
$\sigma=(1,3,5,7,9), \tau=(3,4,5,6,7)$, then $\gamma=(3,5,7)$. We have, $k_0=1,k_1=1,k_2=1,k_3=2$, and $m_0=0,m_1=2,m_2=2,m_3=1$.
\end{example}

A key quantity for the calculation of the variance of the trace estimator is the expectation of the (block) product $\bX \bS \bX^\top$ for any possible row/column 
combination, i.e., quantities of the following type for any $\sigma, \tau\in \pnCycles$:

\begin{align}
\nonumber
\EX[(\bX \bS \bX^\top)_{[\sigma]}&(\bX \bS \bX^\top)_{[\tau]}]
=
\EX\left[\prod_{i=1}^p (\bX \bS \bX^\top)_{\sigma_i, \sigma_{i+1}}\prod_{i=1}^p(\bX \bS \bX^\top)_{\tau_i, \tau_{i+1}}\right]
\\
=& 
\EX\left[\prod_{i=1}^p (X_{\sigma_i}^\top \bS X_{\sigma_{i+1}})\prod_{i=1}^p(X_{\tau_i}^\top \bS X_{\tau_{i+1}})\right]\label{eqn:cross_term}
\\
=&\nonumber
\EX\left[X_{\sigma_1}^\top\bS\left(\prod_{i=2}^p (X_{\sigma_i} X_{\sigma_{i}}^\top\bS)\right)X_{\sigma_1} X_{\tau_1}^\top \bS\left(\prod_{i=2}^p(X_{\tau_i}  X_{\tau_{i}}^\top\bS)\right)X_{\tau_1}
\right]\,.
\end{align}

Here $\EX$ denotes the expected value either with respect to all variables $\bX$ or with respect to all the rows $X_{j:}$ which appear in the product. We shall  write $\EX=\EX_{-t}\EX_{t}$, where $E_t$ is the expectation with respect to the row $\bX_{t:}$ and $\EX_{-t}$ with respect to all the other (independent) rows of  $\bX$.

Another key element of our calculations is the following lemma, which is stated as Lemma D.2 in~\cite{zhang2023trained}.
\begin{lem}
\label{lem:matrix_id}
If $X$ is a Gaussian random vector of $d$ dimension, mean zero and covariance matrix $\Lam$, and $A\in \Real^{d\times d}$ is a fixed matrix, then, 
\begin{align}
\label{eqn:matrix_id}
\EX[X X^\top A X X^\top]=\Lam(A+A^\top)\Lam + \tr(A\Lam)\Lam.
\end{align}
\end{lem}

\section{Trace Lemma}
\label{sec: tracelemma}

 In the sequel, we shall often use the following elementary observation. 
In a product $AX^\top B Y C$ of square matrices $A,B,C$ and vectors $X,Y$, we see that $X^\top B Y$ is a scalar so it is equal to its trace. The cyclic property of the trace operator gives us $\tr(X^\top B Y)=\tr(Y X^\top B)$. Therefore,  $AX^\top B Y C=AC\tr(YX^\top B) $, and $
\tr(AX^\top B Y C)=\tr(AC)\tr(YX^\top B)$. 

\begin{lem}
\label{lem:twocycles}
For any two increasing $p$-cycles, $\sigma$ and $\tau$, and $\gamma=\sigma\cap \tau$ with $|\gamma|=q$, we have, when $q=0$, $
\EX[(\bX \bS \bX^\top)_{[\sigma]}(\bX \bS \bX^\top)_{[\tau]}]
=\tr^2(\bS^p)$. When $q>0$, we have,
\small{
\begin{align}
\label{eqn:twocycles}
\EX[(\bX \bS \bX^\top)_{[\sigma]}(\bX \bS \bX^\top)_{[\tau]}]=\EX\left[\tr\left(\left(\prod_{i=1}^q X_{\gamma_i}X_{\gamma_i}^\top \bS^{k_i}\right)\bS^{k_0}\left(\prod_{i=1}^qX_{\gamma_i}X_{\gamma_i}^\top \bS^{m_i}\right)\bS^{m_0}\right)\right].
\end{align}}
\end{lem}
\begin{proof}   
We consider the two following cases:
\begin{enumerate}
\item 
In the case of $\sigma\cap\tau=\emptyset$, $q=0$, \begin{align*}
&\EX[(\bX \bS \bX^\top)_{[\sigma]}(\bX \bS \bX^\top)_{[\tau]}]
\\
=&\EX\left[X^\top_{\sigma_1}\bS \left(\prod_{i=2}^p X_{\sigma_i}X_{\sigma_i}^\top \bS \right) X_{\sigma_1} X^\top_{\tau_1}\bS \left(\prod_{i=2}^p X_{\tau_i}X_{\tau_i}^\top \bS \right) X_{\tau_1}\right]\\
=&\EX[X^\top_{\sigma_1} \bS^p X_{\sigma_1} X^\top_{\tau_1} \bS^p X_{\tau_1}]=\EX[\tr(X^\top_{\sigma_1} \bS^p X_{\sigma_1}) \tr(X^\top_{\tau_1} \bS^p X_{\tau_1})]
\\=& \EX[X^\top_{\sigma_1} \bS^p X_{\sigma_1}]\EX[X^\top_{\tau_1} \bS^p X_{\tau_1}]=\tr(\bS^p)^2.
\end{align*} 
because of the independence of $X_s$. 
\\ 
In the following we consider the case where $q>0$ where, without loss of generality, we 
assume that $\sigma_1\le \tau_1$. 
\item
When $\sigma_1=\tau_1(=\gamma_1)$, then $k_0=m_0=0$. Consider $X_s$ for $s\in \sigma \Delta \tau$. Every such $X_s$ appears in the form of $X_sX^\top_s$ only once, therefore, the expectation produces an identity matrix, and due to the independence assumption can be omitted. Therefore in this case, 
\small{
\begin{align*}
&\EX[(\bX \bS \bX^\top)_{[\sigma]}(\bX \bS \bX^\top)_{[\tau]}]= \EX\left[X_{\gamma_1}^\top \left(\prod_{i=2}^q X_{\gamma_i}X_{\gamma_i}^\top \bS^{k_i}\right) 
X_{\gamma_1}X_{\gamma_1}^\top \left(\prod_{i=2}^q X_{\gamma_i}X_{\gamma_i}^\top \bS^{m_i}\right)  
X_{\gamma_1}\right].
\end{align*}}
As mentioned before as the expression inside the expectation is a scalar it is equal to its trace.
The cyclic property of the trace operator implies that, 
\small{
\begin{align*}
&\EX[(\bX \bS \bX^\top)_{[\sigma]}(\bX \bS \bX^\top)_{[\tau]}]= \EX\left[\tr\left(X_{\gamma_1}X_{\gamma_1}^\top \left(\prod_{i=2}^q X_{\gamma_i}X_{\gamma_i}^\top \bS^{k_i}\right)  
X_{\gamma_1}X_{\gamma_1}^\top\left(\prod_{i=2}^q X_{\gamma_i}X_{\gamma_i}^\top \bS^{m_i}\right)  \right)
\right], 
\end{align*}}
this is exactly equation~\eqref{eqn:twocycles}. \\
When $\sigma_1<\tau_1$, then $\sigma_1\notin \gamma$.
\item
 First, consider $\tau_1=\gamma_1\in \gamma$, thus $m_0=0$, we have, 
\small{
\begin{align*}
&\EX[(\bX \bS \bX^\top)_{[\sigma]}(\bX \bS \bX^\top)_{[\tau]}]= \EX\left[X_{\sigma_1}^\top \bS^{k_0}\left(\prod_{i=1}^q X_{\gamma_i}X_{\gamma_i}^\top \bS^{k_i}\right)  
X_{\sigma_1}X_{\gamma_1}^\top \bS^{m_1}\left(\prod_{i=2}^q X_{\gamma_i}X_{\gamma_i}^\top \bS^{k_i}\right)
X_{\gamma_1}\right] 
\end{align*}}
\small{
\begin{align*}
&\EX[(\bX \bS \bX^\top)_{[\sigma]}(\bX \bS \bX^\top)_{[\tau]}]
\\ = & \EX\left[X_{\sigma_1}^\top \bS^{k_0}\left(\prod_{i=1}^q X_{\gamma_i}X_{\gamma_i}^\top\bS^{k_i}\right) X_{\sigma_1} X_{\gamma_1}^\top \bS^{m_1}\left(\prod_{i=2}^q X_{\gamma_i}X_{\gamma_i}^\top \bS^{m_i}\right) 
X_{\gamma_1}\right] 
\\ = & \EX\left[\EX_{\sigma_1}\left[X_{\sigma_1}^\top \bS^{k_0}\left(\prod_{i=1}^q X_{\gamma_i}X_{\gamma_i}^\top\bS^{k_i}\right) X_{\sigma_1} \right]X_{\gamma_1}^\top \bS^{m_1}\left(\prod_{i=2}^q X_{\gamma_i}X_{\gamma_i}^\top \bS^{m_i}\right) 
X_{\gamma_1}\right] 
\\\stackrel{(I)}{=}&\EX\left[\tr\left(\bS^{k_0}\left(\prod_{i=1}^q X_{\gamma_i}X_{\gamma_i}^\top\bS^{k_i}\right)\right)X_{\gamma_1}^\top \bS^{m_1}\left(\prod_{i=2}^q X_{\gamma_i}X_{\gamma_i}^\top \bS^{m_i}\right)
X_{\gamma_1}\right]
\\\stackrel{(II)}{=}&\EX\left[\tr\left(\left(\prod_{i=1}^q X_{\gamma_i}X_{\gamma_i}^\top\bS^{k_i}\right)\bS^{k_0}\right)X_{\gamma_1}^\top \bS^{m_1}\left(\prod_{i=2}^q X_{\gamma_i}X_{\gamma_i}^\top \bS^{m_i}\right)
X_{\gamma_1}\right]
\\=&\EX\left[\tr\left(X_{\gamma_1}X_{\gamma_1}^\top\bS^{k_1}\left(\prod_{i=2}^q X_{\gamma_i}X_{\gamma_i}^\top\bS^{k_i}\right)\bS^{k_0}\right)X_{\gamma_1}^\top \bS^{m_1}\left(\prod_{i=2}^q X_{\gamma_i}X_{\gamma_i}^\top \bS^{m_i}\right) 
X_{\gamma_1}\right]
\\\stackrel{(III)}{=}&\EX\left[\tr\left(X_{\gamma_1}^\top\bS^{k_1}\left(\prod_{i=2}^q X_{\gamma_i}X_{\gamma_i}^\top\bS^{k_i}\right)\bS^{k_0}X_{\gamma_1}\right)X_{\gamma_1}^\top \bS^{m_1}\left(\prod_{i=2}^q X_{\gamma_i}X_{\gamma_i}^\top \bS^{m_i}\right)
X_{\gamma_1}\right]
\\\stackrel{(IV)}{=}&\EX\left[X_{\gamma_1}^\top\bS^{k_1}\left(\prod_{i=2}^q X_{\gamma_i}X_{\gamma_i}^\top\bS^{k_i}\right)\bS^{k_0}X_{\gamma_1}X_{\gamma_1}^\top \bS^{m_1}\left(\prod_{i=2}^q X_{\gamma_i}X_{\gamma_i}^\top \bS^{m_i}\right)
X_{\gamma_1}\right]\\\stackrel{(V)}{=}&\EX\left[\tr\left(X_{\gamma_1}X_{\gamma_1}^\top\bS^{k_1}\left(\prod_{i=2}^q X_{\gamma_i}X_{\gamma_i}^\top\bS^{k_i}\right)\bS^{k_0}X_{\gamma_1}X_{\gamma_1}^\top \bS^{m_1}\left(\prod_{i=2}^q X_{\gamma_i}X_{\gamma_i}^\top \bS^{m_i}\right) \right)
\right]
\\=&\EX\left[\tr\left(\left(\prod_{i=1}^q X_{\gamma_i}X_{\gamma_i}^\top\bS^{k_i}\right)\bS^{k_0}\left(\prod_{i=1}^q X_{\gamma_i}X_{\gamma_i}^\top \bS^{m_i}\right) \right)
\right],
\end{align*}}
where (I) follows from \eqref{eqn:sampletrace}, (II), (III) and (V) the results of the cyclic property of the trace operator, and (IV) equates a scalar to its trace.
This is again in the desired form. 
\item
Finally, we consider the case of $\tau_1\notin \sigma\cap\tau$, $\sigma_1<\tau_1<\gamma_1$, through similar arguments. We have,
\begin{align*}
&\EX[(\bX \bS \bX^\top)_{[\sigma]}(\bX \bS \bX^\top)_{[\tau]}]= 
\\=&
\EX\left[\tr\left(\bS^{k_0}\prod_{i=1}^q X_{\gamma_i}X_{\gamma_i}^\top\bS^{k_i}\right)\tr\left(\bS^{m_0}\prod_{i=1}^q X_{\gamma_i}X_{\gamma_i}^\top\bS^{m_i} \right)\right] ,\qquad
\text{after taking $\EX_{\sigma_1}$ and $\EX_{\tau_1}$}
\\=&
\EX\left[\tr\left(\prod_{i=1}^q X_{\gamma_i}X_{\gamma_i}^\top\bS^{k_i}\bS^{k_0}\right)\tr\left(\prod_{i=1}^q X_{\gamma_i}X_{\gamma_i}^\top\bS^{m_i} \bS^{m_0}\right)\right] 
\\=&
\EX\left[\tr\left(X_{\gamma_1}X_{\gamma_1}^\top\bS^{k_1}\prod_{i=2}^qX_{\gamma_i}X_{\gamma_i}^\top\bS^{k_i}\bS^{k_0}\right)\tr\left(X_{\gamma_1}X_{\gamma_1}^\top\bS^{m_1}\prod_{i=2}^q X_{\gamma_i}X_{\gamma_i}^\top\bS^{m_i} \bS^{m_0}\right)\right] 
\\=&
\EX\left[\tr\left(X_{\gamma_1}^\top\bS^{k_1}\prod_{i=2}^q X_{\gamma_i}X_{\gamma_i}^\top\bS^{k_i}\bS^{k_0}X_{\gamma_1}\right)\tr\left(X_{\gamma_1}^\top\bS^{m_1}\prod_{i=2}^qX_{\gamma_i}X_{\gamma_i}^\top\bS^{m_i} \bS^{m_0}X_{\gamma_1}\right)\right] 
\\=&
\EX\left[X_{\gamma_1}^\top\bS^{k_1}\prod_{i=2}^q X_{\gamma_i}X_{\gamma_i}^\top\bS^{k_i}\bS^{k_0}X_{\gamma_1}X_{\gamma_1}^\top\bS^{m_1}\prod_{i=2}^q X_{\gamma_i}X_{\gamma_i}^\top\bS^{m_i} \bS^{m_0}X_{\gamma_1}\right] 
\\=&
\EX\left[\tr\left(X_{\gamma_1}^\top\bS^{k_1}\prod_{i=2}^qX_{\gamma_i}X_{\gamma_i}^\top\bS^{k_i}\bS^{k_0}X_{\gamma_1}X_{\gamma_1}^\top\bS^{m_1}\prod_{i=2}^qX_{\gamma_i}X_{\gamma_i}^\top\bS^{m_i} \bS^{m_0}X_{\gamma_1}\right)\right] 
\\=&
\EX\left[\tr\left(X_{\gamma_1}X_{\gamma_1}^\top\bS^{k_1}\prod_{i=2}^qX_{\gamma_i}X_{\gamma_i}^\top\bS^{k_i}\bS^{k_0}X_{\gamma_1}X_{\gamma_1}^\top\bS^{m_1}\prod_{i=2}^q X_{\gamma_i}X_{\gamma_i}^\top\bS^{m_i} \bS^{m_0}\right)\right].
\end{align*}
Thus, equation~\eqref{eqn:twocycles} holds. 
\end{enumerate}
This concludes the proof of Lemma~\ref{lem:twocycles}.
\end{proof}

\section{Recursion Structure of Trace Calculations}
\label{sec:MNCalc}

For each integer $q\ge 1$, and vectors of non-negative integers $\bk_q=( k_1, k_2, \ldots k_q)$ and $\bm_q=(m_1, m_2, \ldots m_q)$, define,
\small{
\begin{align*}
M(q;k_1, k_2, \ldots k_q\big|m_1, m_2, \ldots m_q) := \EX\left[\tr\left(\prod_{i=1}^q(X_i X_i^\top) \bS^{k_i}\prod_{i=1}^q(X_i X_i^\top) \bS^{m_i} \right)\right],
\end{align*}}
and
\small{
\begin{align*}
N(q;k_1, k_2, \ldots k_q\big|m_q, m_{q-1} \ldots m_1):= \EX\left[\tr\left(\prod_{i=1}^q(X_i X_i^\top) \bS^{k_i}\prod_{i=1}^q(X_{q+1-i} X_{q+1-i}^\top) \bS^{m_{q+1-i}} \right)\right].
\end{align*}}
Here $X_i,i=1,2,\ldots, q $ are i.i.d $d$-variate standard Gaussian random variables. As indicated in~\eqref{eqn:twocycles}, the evaluation of $M(q;q_1, q_2, \ldots k_q\big|m_1, m_2, \ldots m_q)$ is key to the estimation of the variance.

\begin{lem}
\label{lem:recursion}
For $q=1$, we have, $M(1;k\big|m)=N(1;k\big|m)=2\tr(\bS^{k+m})+\tr(\bS^{k})\tr(\bS^{m})$. 
For $q\ge 1$, we have the following relationships,
\begin{align}
M&(q+1;k_1, k_2, \ldots k_q, k_{q+1}\big| m_1, m_2, \ldots m_q, m_{q+1})\nonumber
\\= &
M(q; k_1, k_2, \ldots k_q+k_{q+1}\big| m_1, m_2, \ldots m_q+m_{q+1})  \nonumber
\\&
+
N(q; k_1, k_2, \ldots k_q+m_{q}, \big|m_{q-1}, \ldots m_1, k_{q+1}+m_{q+1}) \nonumber
\\&
+
M(q; k_1, k_2, \ldots k_q+m_{q+1}\big| m_1, m_2, \ldots m_q+k_{q+1}),
\label{eqn:recusion_M}
\end{align}
\begin{align}
N&(q+1;k_1, k_2, \ldots k_q, k_{q+1}\big| m_{q+1}, m_q, \ldots m_2, m_{1})\nonumber\\ =&  2\,N(q; k_1, k_2, \ldots k_q+k_{q+1}+m_{q+1}\big| m_q, \ldots m_2, m_1)
\nonumber \\& +
\tr[\bS^{k_{q+1}}]N(q; k_1, k_2, \ldots k_q+m_{q+1}\big| m_q, \ldots m_2, m_{1}).\label{eqn:recusion_N}
\end{align}
\end{lem}
\begin{proof}
For $q=1$, and any two non-negative integers $k$ and $m$, following Lemma~\ref{lem:matrix_id} results to
\begin{align*}
M(1; k\big|m)=N(1;k\big|m)=&\EX[\tr(X_1X_1^\top \bS^{k}X_1X_1^\top\bS^{m})]
= 2\tr(\bS^{k+m})+\tr(\bS^k)\tr(\bS^m).
\end{align*}
Moreover, from the cycle property of the trace operator, we can write 
\begin{align*}
&M(q+1;k_1, k_2, \ldots k_q, k_{q+1}\big| m_1, m_2, \ldots m_q, m_{q+1})
\\
=&
\EX\left[\tr\left( \prod_{i=1}^{q}(X_i X_i^\top) \bS^{k_i}(X_{q+1} X_{q+1}^\top) \bS^{k_{q+1}}\prod_{i=1}^{q}(X_i X_i^\top) \bS^{m_i}(X_{q+1} X_{q+1}^\top) \bS^{m_{q+1}}\right)\right]
\\
=&
\EX\left[\tr\left((X_{q+1} X_{q+1}^\top) \bS^{k_{q+1}}\prod_{i=1}^{q}(X_i X_i^\top) \bS^{m_i}(X_{q+1} X_{q+1}^\top) \bS^{m_{q+1}} \prod_{i=1}^{q}(X_i X_i^\top) \bS^{k_i}\right)\right].
\end{align*}
Taking expectation with respect to $X_{q+1}$, and applying Lemma~\ref{lem:matrix_id}, we can write the above quantity in the following equivalent expression:
\begin{align*}
& \EX\left[\tr\left( \left[\bS^{k_{q+1}}\prod_{i=1}^{q}(X_i X_i^\top) \bS^{m_i} \right]\bS^{m_{q+1}} \prod_{i=1}^{q}(X_i X_i^\top) \bS^{k_i}\right)\right]
\\& + \EX\left[\tr\left( \left[(\bS^{k_{q+1}}\prod_{i=1}^{q}(X_i X_i^\top) \bS^{m_i})^\top \right]\bS^{m_{q+1}} \prod_{i=1}^{q}(X_i X_i^\top) \bS^{k_i}\right)\right]
\\ &+ 
\EX\left[\tr\left( \tr\left(\bS^{k_{q+1}}\prod_{i=1}^{q}(X_i X_i^\top) \bS^{m_i}\right)\bS^{m_{q+1}} \prod_{i=1}^{q}(X_i X_i^\top) \bS^{k_i}\right)\right].
\end{align*}
The recursion~\eqref{eqn:recusion_M} follows from simplification of these three terms. Similarly, we have, 
\begin{align*}
&N(q+1;k_1, k_2, \ldots k_q, k_{q+1}\big| m_{q+1}, m_q, \ldots m_2, m_1)
\\ & \quad =
\EX\left[\tr\left(\prod_{i=1}^{q+1}(X_i X_i^\top) \bS^{k_i}\prod_{i=1}^{q+1}(X_{q+2-i} X_{q+2-i}^\top) 
\bS^{m_{q+2-i}} \right)\right] 
\\& \quad =
\EX\left[\tr\left(\prod_{i=1}^q(X_i X_i^\top) \bS^{k_i}(X_{q+1} X_{q+1}^\top) \bS^{k_{q+1}}(X_{q+1} X_{q+1}^\top) \bS^{m_{q+1}}\prod_{i=2}^{q+1}(X_{q+2-i} X_{q+2-i}^\top) \bS^{m_{q+2-i}} \right)\right]
\\& \quad =
\EX\left[\tr\left(\prod_{i=1}^q(X_i X_i^\top) \bS^{k_i}\EX[(X_{q+1} X_{q+1}^\top) \bS^{k_{q+1}}(X_{q+1} X_{q+1}^\top)] \bS^{m_{q+1}}\prod_{i=2}^{q+1}(X_{q+2-i} X_{q+2-i}^\top) \bS^{m_{q+2-i}} \right)\right]
\\& \quad \stackrel{(I)}{=}
2\,\EX\left[\tr\left(\prod_{i=1}^q(X_i X_i^\top) \bS^{k_i}\bS^{k_{q+1}}\bS^{m_{q+1}}\prod_{i=2}^{q+1}(X_{q+2-i} X_{q+2-i}^\top) \bS^{m_{q+2-i}} \right)\right] +
\\&\qquad + \tr[\bS^{k_{q+1}}]\EX\left[\tr\left(\prod_{i=1}^q(X_i X_i^\top) \bS^{k_i}\bS^{m_{q+1}}\prod_{i=2}^{q+1}(X_{q+2-i} X_{q+2-i}^\top) \bS^{m_{q+2-i}} \right)\right]
\\& \quad =
2\,N(q; k_1, k_2, \ldots k_q+k_{q+1}+m_{q+1}\big|m_q, \ldots m_2, m_1)
\\&\qquad +
\tr[\bS^{k_{q+1}}]N(q; k_1, k_2, \ldots k_q+m_{q+1}\big|m_q, \ldots m_2, m_{1}),
\end{align*}
where Lemma~\ref{lem:matrix_id} is applied to get (I).
\end{proof}

\subsection{Calculations of $N(q, \cdot, \cdot)$}
\label{sec:calc_N_New}

Consider once more the recursion,  
\begin{align*}
&N(q+1:k_1, k_2, \ldots, k_q, k_{q+1} \big| m_{q+1}, m_q, \ldots, m_2, m_1)\nonumber 
\\ & \quad =
2\,N(q;k_1, k_2, \ldots, k_q+k_{q+1}+m_{q+1} \big| m_q, \ldots, m_2, m_1)+\nonumber 
\\ & \qquad +  
S_{k_{q+1}}N(q: k_1, k_2, \ldots, k_q+m_{q+1} \big| m_q, \ldots, m_2, m_1). 
\end{align*}
Expanding this quantity gives
\begin{align*}
&N(q+1; k_1, k_2, \ldots, k_q, k_{q+1} \big| m_{q+1}, m_q, \ldots, m_2, m_1)\nonumber 
\\ & =
4\,N(q-1; k_1, k_2, \ldots, k_{q-2}, (k_{q-1}+m_q)+(k_q+m_{q+1})+k_{q+1} \big| m_{q-1}, \ldots, m_2, m_1)+
\\ &  +  
2\,S_{(k_q+m_{q+1})+k_{q+1}}N(q-1; k_1, k_2, \ldots, k_{q-2}, (k_{q-1}+m_q)\big| m_{q-1}, \ldots, m_2, m_1)
\\&+2S_{k_{q+1}}N(q-1, k_1, k_2, \ldots, k_{q-2}, (k_{q-1}+m_q)+(k_q+m_{q+1})\big| m_{q-1}, \ldots, m_2, m_1)
\\&+S_{k_{q+1}}S_{(k_q+m_{q+1})}N(q-1, k_1, k_2, \ldots, k_{q-2}, (k_{q-1}+m_q)\big| m_{q-1}, \ldots, m_2, m_1). 
\end{align*}
\subsubsection*{Introducing the symbol {\rm$\ostr$} and a related algebraic structure.}
The recursive formula of $N$ prompts us to introduce the following notation. 
 Let 
 \[
 W=\left\{2^r s_{i_0}s_{i_1}\dots s_{i_q} | r,q\in\{ 0,1,\dots,\}, i_n\in\Nat \}\right\}
 \]
 be a set of finite words with the alphabet $\{s_i, i\in \Nat\}$ preceded by a power of 2. We define two operations on $W$ generated by
$s_i\oplus s_j=2s_{i+j}$ and $s_i\otimes s_j=s_is_j$. 
Namely, if $w=2^r s_{i_0}\dots s_{i_k}$ and
$v=2^t s_{j_0}\dots s_{j_l}$ then
\begin{align*}
w\oplus v&=2^{r+t+1}s_{i_0}\dots s_{i_k+j_0}s_{j_1}\dots s_{j_l},
\\
w\otimes v
&=2^{r+t}s_{i_0}\dots s_{i_k}s_{j_0}s_{j_1}\dots s_{j_k}\,.
\end{align*}
For example, $s_i\oplus s_j \otimes s_k\oplus s_\ell= 4s_{i+j}s_{k+\ell}$. 
We shall use $\ostr$ to represent the generic operation, which means that it could be either $\otimes$ or $\oplus$. The operation $\ostr$ is associative $(w\ostr v)\ostr u=w\ostr(v\ostr v)$, regardless of the choice of the operation.
After resolving all $\ostr$ in the expression 
$w=s_{i_0}\ostr s_{i_1}\ostr \dots\ostr s_{i_q}$
we get $w=2^r s_{j_0}s_{j_1}\dots s_{j_k}$,
where the sum of indices remains unchanged $\sum_{\ell=0}^k j_\ell=\sum_{\ell=0}^q i_\ell$, 
the exponent $r$ is the number $\#\oplus$ of operations $\oplus$  among the $\ostr$'s and $k$, the length of the word, is the number $\#\otimes$ of operations $\otimes$ among $\ostr$'s in the original expression.

 Given the evaluation of a letter $\eV:s_i\mapsto S_i\in\Real$, we can extend it to the case of words $w=2^r s_{i_0}s_{i_1}\dots s_{i_q}$ as  $\eV(w)=2^r\cdot\prod_{\ell=0}^q S_{i_\ell}$. For a sequence of words $w_i$ we define
\begin{align}
\sum_\ostr \eV(w_1\ostr w_2\ostr\dots\ostr w_k)
    \label{eqn:sum ostar}
\end{align}
as the sum over all the $2^{q}$ choices of $\oplus$ and $\otimes$ in place of all $\ostr$'s.


For $q\ge 1$, let 
\begin{align*}
\al^q_0:=m_1, \quad \al^q_i=k_i+m_{i+1}, i=1,\ldots, q-1, \quad \al^q_q=k_q.
\end{align*}
\begin{lem}
\label{lem:N_in_beta}
With the evaluation $\eV(s_i)=S_i=\tr[\bS^i]$ we have, for $q\ge 1$,
{\rm
\begin{align}
&N(q; k_1, k_2, \ldots, k_q | m_q, \ldots, m_2, m_1)
= \sum_\ostr 
\eV(s_{\al^q_0}\ostar s_{\al^q_1}\ostar\ldots \ostar s_{\al^q_q})\,.
\label{eqn:N_in_beta}
\end{align}
}
\end{lem}
We note that although the evaluation does not depend on the order of the letters in the final word, the word itself depends very much on the order of the operations $\ostar$ between the letters.
\begin{proof}
The recursion~\eqref{eqn:recusion_N} provide exactly the induction for the two possibilities 
of adding one more operation. 
\end{proof}

\subsection{Calculations of $M(q, \cdot, \cdot)$}
\label{sec:calc_M}
For any natural numbers $q$ and $k_1, k_2, \ldots k_q, m_1, m_2, \ldots m_q$, define
\begin{align}\label{def:KMij}
K_{i,j}:= \sum_{\ell=i}^j k_\ell, \quad M_{i,j}:= \sum_{\ell=i}^j m_\ell, \forall, i\le j, 
\end{align}
and $K_{i,j}=M_{i,j}=0$ for all $i>j$. We use the notation $K_{i,j}^{\sigma,\tau}$
and $M_{i,j}^{\sigma,\tau}$ when $k_i=k^{\sigma,\tau}_i$ and $m_i=m^{\sigma,\tau}_i$.

\begin{lem}
\label{lem:expression_M}
For any $q >1$ we can write:
\begin{align}
&M(q;k_1, k_2, \ldots k_q \big| m_1, m_2, \ldots m_q)\nonumber \\ & \quad = \sum_{t=1}^{q-1} 2^{t-1}N(q-t;k_1, \ldots, k_{q-t}+m_{q-t} \big| m_{q-t-1}, \ldots, m_1, K_{q-t+1,q}+M_{q-t+1,q})
\label{eqn:expression_M}
\end{align}
\end{lem}
\begin{proof}[Proof by induction] Recall the recursion in~\eqref{eqn:recusion_M}
\begin{align*}
&M(q+1;k_1, k_2, \ldots k_q, k_{q+1} \big| m_1, m_2, \ldots m_q, m_{q+1})\nonumber
\\ & \quad = 
M(q; k_1, k_2, \ldots k_q+k_{q+1} \big| m_1, m_2, \ldots m_q+m_{q+1})+\nonumber
\\ & \qquad +  
N(q; k_1, k_2, \ldots k_q+m_{q}\big| m_{q-1}, \ldots m_1, k_{q+1}+m_{q+1})+\nonumber\\ & \qquad +  M(q; k_1, k_2, \ldots k_q+m_{q+1} \big| m_1, m_2, \ldots m_q+k_{q+1}).
\end{align*}
Apply the induction assumption, we have, 
\begin{align*}
&M(q+1;k_1, k_2, \ldots k_q, k_{q+1}\big| m_1, m_2, \ldots m_q, m_{q+1})
\\ & \quad = 
\sum_{t=1}^{q-1} 2^{t-1}N(q-t;k_1, \ldots, k_{q-t}+m_{q-t} \big| m_{q-t-1}, \ldots, m_1, K_{q-t+1,q+1}+M_{q-t+1,q+1})+
\\ & \qquad +
\sum_{t=1}^{q-1} 2^{t-1}N(q-t; k_1, \ldots, k_{q-t}+m_{q-t} \big| m_{q-t-1}, \ldots, m_1, K_{q-t+1,q+1}+M_{q-t+1,q+1})+
\\ & \qquad +  
N(q; k_1, k_2, \ldots k_q+k_{q+1} \big| m_q, m_{q-1}, \ldots m_1+m_{q+1})
\\ & \quad =
\sum_{t=1}^{q} 2^{t-1}N(q-t;k_1, \ldots, k_{q-t}+m_{q-t} \big| m_{q-t-1}, \ldots, m_1, K_{q-t+1,q+1}+M_{q-t+1,q+1}).
\end{align*}
\end{proof}

\section{Main Results}

\subsection{Calculations of the Variance}

Leveraging the theoretical results presented so far, we can now calculate the variance 
$V_p^n$ explicitly using the notation of~\eqref{eqn:sum ostar}.

\begin{thm}[Representation of the variance]\label{thm:varrep}
The variance of $V_p^n$  is equal to
{\rm
\begin{align*}
&\binom{n}{p}^{-2}\left[\binom{n}{2p}\binom{2p}{p}\tr^2(\bS^p)+\binom{n}{2p-1}\binom{2p-1}{1}\binom{2p-2}{p-1}[2\tr(\bS^{2p})+\tr^2(\bS^p)]\right.
\\ & \left.+\sum_{\sigma, \tau\in \pnCycles, |\gamma|\ge 2}\sum_{t=1}^{q-1} 2^{t-1}\sum_{\ostar}
\eV(s_{\beta^{q-t}_0(\sigma, \tau)}\ostar s_{\beta^{q-t}_1(\sigma, \tau)}\ostar\ldots \ostar s_{\beta^{q-t}_{q-t}(\sigma, \tau)})\right]-\tr^2(\bS^p),
\end{align*}
}
where $\beta_0^{q-t}(\sigma, \tau):= K_{q-t+1,q}^{\sigma, \tau}+M_{q-t+1,q}^{\sigma, \tau}$, 
$\beta_i^{q-t}(\sigma, \tau)=k^{\sigma, \tau}_i+m^{\sigma, \tau}_{i}, i=1,\ldots, q-t-1$, 
$\beta^{q-t}_{q-t}(\sigma, \tau)= k_{q-t}^{\sigma, \tau}+m_{q-t}^{\sigma, \tau}$; by convention $\binom{n}{m}=0$ for $n<m$.
\end{thm}
\begin{proof}
The desired variance is of the form 
\begin{align*}
\binom{n}{p}^{-2}\sum_{\sigma, \tau\in \pnCycles}\EX[(\bX \bS \bX^\top)_{[\sigma]}(\bX \bS \bX^\top)_{[\tau]}]-\tr^2(\bS^p).
\end{align*}
Let $q$ denote the cardinality of $\gamma=\sigma\cap \tau$. When $q=0$, 
$\EX[(\bX \bS \bX^\top)_{[\sigma]}(\bX \bS \bX^\top)_{[\tau]}]=\tr^2(\bS^p)$, and there are $\binom{n}{2p}\binom{2p}{p}$ such terms. When $q=1$, 
$\EX[(\bX \bS \bX^\top)_{[\sigma]}(\bX \bS \bX^\top)_{[\tau]}]=2\tr(\bS^{2p})+\tr^2(\bS^p)$, and there are $\binom{n}{2p-1}\binom{2p-1}{1}\binom{2p-2}{p-1}$ such terms. For $q\ge 2$, as the consequence of Lemmata~\ref{lem:twocycles},~\ref{lem:recursion}, ~\ref{lem:N_in_beta} and~\ref{lem:expression_M}, we have,
\begin{align*}
&\EX[(\bX \bS \bX^\top)_{[\sigma]}(\bX \bS \bX^\top)_{[\tau]}]\\
=&M(q;k^{\sigma, \tau}_1, k_2^{\sigma, \tau}, \ldots k_q^{\sigma, \tau}+k_0^{\sigma, \tau}\big|m_1^{\sigma, \tau}, m_2^{\sigma, \tau}, \ldots m_q^{\sigma, \tau}+m_0^{\sigma, \tau}) \\ =&\sum_{t=1}^{q-1} 2^{t-1}N(q-t;k_1^{\sigma, \tau}, \ldots, k_{q-t}^{\sigma, \tau}+m_{q-t}^{\sigma, \tau} \big| m_{q-t-1}^{\sigma, \tau}, \ldots, m_1^{\sigma, \tau}, K_{q-t+1,q}^{\sigma, \tau}+M_{q-t+1,q}^{\sigma, \tau})
\\ =&\sum_{t=1}^{q-1} 2^{t-1}\sum_{\ostar}
\eV(s_{\beta^{q-t}_0(\sigma, \tau)}\ostar s_{\beta^{q-t}_1(\sigma, \tau)}\ostar\ldots \ostar s_{\beta^{q-t}_{q-t}(\sigma, \tau)}).
\end{align*}
\end{proof}

\subsection{Bounds on the Variance}

Let us start with an example for $p=2$. Note that $p=1$ is the case of Frobenius norm, which has been extensively studied. 

\begin{example}
\label{exm:p=2}
The estimated average is $\tr(\bS^2)$. Now consider the second moment of estimate consists of: 1) $\binom{n}{4}\binom{4}{2}$ terms of two non-overlapping increasing cycles; 2) $\binom{n}{3}3!$ terms of two increasing cycles with one common index; 3) $\binom{n}{2}$ terms with two same increasing cycles. Easy to see that, for 1), each term equals to $\tr^2(\bS^2)$, for 2) each term equals to $M(1;1|1)=2\tr(\bS^4)+\tr^2(\bS^2)$ and for 3), each term equals to $M(2;1,1|1,1)=6\tr(\bS^4)+3\tr^2(\bS^2)$. 
Put them together, we know that, the second moment equals,
\begin{align*}
\binom{n}{2}^{-2} \left\{ \binom{n}{4}\binom{4}{2}\tr^2(\bS^2)+ \binom{n}{3}3![2\tr(\bS^4)+\tr^2(\bS^2)]+\binom{n}{2}[6\tr(\bS^4)+3\tr^2(\bS^2)]\right\}.
\end{align*}
Meanwhile, we know that, 
\begin{align*}
\tr^2(\bS^2)/d\le \tr(\bS^4)\le \tr^2(\bS^2),
\end{align*}
with the first inequality being the result of the Cauchy-Schwarz inequality and the second one follows from the non-negativity of the cross term in a complete square. Thus,
the second moment is lower and upper bounded by,
\begin{align*}
\left\{\binom{n}{2}^{-2} \left[\binom{n}{3}\frac{2 \cdot 3!}{d}+\binom{n}{2}\left(\frac{6}{d} +2\right)\right] +1\right\}\tr^2(\bS^2)
=\left[\frac{4(2n+d-1)}{dn(n-1)}+1\right]\tr^2(\bS^2)\end{align*}
and 
\begin{align*}
\left\{\binom{n}{2}^{-2} \left[2\binom{n}{3} 3!+8\binom{n}{2} \right] +1\right\}\tr^2(\bS^2)=\left[\frac{8n}{n(n-1)}+1\right]\tr^2(\bS^2),
\end{align*}
respectively. Both bounds means that the variance is in the order of $O(\frac{1}{n})$. 
\end{example}

In general, as a consequence of the H\"older's inequality, for $0< \beta \le p$, we have
\begin{align}
\label{eqn:trace_relation_1}
\tr^{\frac{\beta}{p}}(\bS^p)\le \tr(\bS^\beta) \le \tr^{\frac{\beta}{p}}(\bS^p) d^{1-\frac{\beta}{p}}.
\end{align}
Similarly, for $\beta \ge p$, we have, 
\begin{align}
\label{eqn:trace_relation_2}
\tr^{\frac{\beta}{p}}(\bS^p) d^{1-\frac{\beta}{p}}\le \tr(\bS^\beta) \le \tr^{\frac{\beta}{p}}(\bS^p).
\end{align}
These inequalities lead to the following bound on $\eV(s_{\beta_0}\ostr s_{\beta_1}\ostr \ldots \ostr s_{\beta_q})$.
\begin{lem}
\label{lem:bound_on_N}
For any $1\le q \le p$, if $\al_j\le p$, $j=0,1,\ldots,q$
\begin{align}
\label{eqn:bound_on_N}
S_p^2  3^q\le N(q, k_1,k_2, \ldots, k_q \big| m_1,m_2, \ldots, m_q)
\le S_p^2 d^{q-1}3^q.
\end{align}
Otherwise, we have, 
\begin{align}
\label{eqn:bound_on_N_1}
S_p^2  3^q/d\le N(q, k_1,k_2, \ldots, k_q \big| m_1,m_2, \ldots, m_q)
\le S_p^2 d^{q}3^q.
\end{align}
\end{lem}
\begin{proof}
In the case that $\al_j\le p$, $j=0,1,\ldots,q$, inequality~\eqref{eqn:trace_relation_1} and the fact that $s_{\al_i}\oplus s_{\al_j}\le 2s_{\al_i}\otimes s_{\al_j}$ imply that,
\begin{align*}
2^{\text{\# of $\oplus$}}\tr^2(\bS^p)\le \eV(s_{\al_0}\ostr s_{\al_1}\ostr \ldots \ostr s_{\al_q}) \le d^{q-1}2^{\text{\# of $\oplus$}}\tr^2(\bS^p).
\end{align*}
Should one of the $\al_j>p$, then, the above arguments leads to
\begin{align*}
d^{-1}2^{\text{\# of $\oplus$}}\tr^2(\bS^p)\le \eV(s_{\al_0}\ostr s_{\al_1}\ostr \ldots \ostr s_{\al_q} )\le d^{q}2^{\text{\# of $\oplus$}}\tr^2(\bS^p).
\end{align*}
The bounds~\eqref{eqn:bound_on_N} and~\eqref{eqn:bound_on_N_1} follow from the summation expression in Lemma~\ref{lem:N_in_beta}.
\end{proof}
\begin{lem}
\label{lem:bound_on_M}
For $q\le p/2$, we have, 
\begin{align}
\label{eqn:bound_on_M}
\frac{3^q-2^2}{d^2}S_p^2 \le M(q, k_1,k_2, \ldots, k_q \big| m_1,m_2, \ldots, m_q)
\le S_p^2 3^{q-1}d^{q-1}\frac{1-\left(\frac{2}{3d}\right)^{q-1}}{1-\frac{2}{3d}}.
\end{align}
Otherwise, we have
\begin{align}
\label{eqn:bound_on_M_1}
\frac{3^q-2^2}{d}S_p^2 \le M(q, k_1,k_2, \ldots, k_q \big| m_1,m_2, \ldots, m_q)
\le S_p^2 3^{q-1}d^{q-2}\frac{1-\left(\frac{2}{3d}\right)^{q-1}}{1-\frac{2}{3d}}.
\end{align}
\end{lem}
\begin{proof}
From~\eqref{eqn:N_in_beta}, we know that, in the expression in Lemma~\ref{lem:N_in_beta}, $\al_j\ge 2$, $j=1,\ldots,q-1$. Meanwhile, recursion~\eqref{eqn:recusion_M} indicates that $\al_0$ and $\al_q$ will be at least $2$ as well. So, when $q>p/2$, there is no $\al_i$ that can exceed $p$. Hence, the bound~\eqref{eqn:bound_on_M_1} follows from  summing over the bounds~\eqref{eqn:bound_on_N_1} according to the expression in Lemma~\ref{lem:expression_M}. Otherwise, the bounds~\eqref{eqn:bound_on_N_1} will be summed over to get inequality~\eqref{eqn:bound_on_M_1}. 
\end{proof}

\begin{thm}
\label{thm:upper_bound}
The variance can be upper bounded by $(B_1+B_2+B_3+B_4)\tr^2(\bS^p)$, with
\begin{align*}
B_1=&\left\{\begin{array}{cc}\prod_{k=0}^{p-1}\frac{n-k-p}{n-k}-1, & n\ge 2p\\ 0, & n <2p, \end{array}\right. 
\\
B_2=&\left\{\begin{array}{cc}\frac{(n-p)!(n-p)!}{n!(n-2p)!}\frac{p^2}{n-2p+1}, & n\ge 2p\\ 0, &n <2p,\end{array}\right.
\\
B_3=&\frac{2}{3d^2}\left[\left(\frac{3pd}{n}+1\right)^p- \frac{3p^2d}{n}-1\right],
\\
B_4=&\frac{2^p(d-1)}{3d^2}d^{p/2}\left(\frac{3p}{n}\right)^{p/2}.
\end{align*}
\end{thm}
\begin{proof}
The number of terms that has $q$ common indices,
$q=0,1,\ldots,p$, between two cycles equals 
\begin{align*}
\binom{n}{2p-q}\binom{2p-q}{q}\binom{2p-2q}{p-q},
\end{align*}
which, according to Lemma~\ref{lem:tech1},  equals $\binom{n-p}{p-q}\binom{p}{q}\binom{n}{p}$. Therefore, Lemma~\ref{lem:bound_on_M} implies that the variance is upper bounded by, with $q=|\gamma|$,
\begin{align*}
&\binom{n}{p}^{-2}\left[\binom{n}{2p}\binom{2p}{p}\tr^2(\bS^p)+\binom{n}{2p-1}\binom{2p-1}{1}\binom{2p-2}{p-1}(2\tr(\bS^{2p})+\tr^2(\bS^p))\right.
\\ & \left.+\left(\sum_{ 2\le q\le p/2}d^{q-1}+ \sum_{p/2< q\le p}d^{q-2}\right)3^{q-1}\left(\frac{1-\left(\frac{2}{3d}\right)^{q-1}}{1-\frac{2}{3d}}\right)\frac{(d+2)^q-2^q}{d}\binom{n}{2p-q}\binom{2p-q}{q}\binom{2p-2q}{p-q}\tr^2(\bS^p)\right]-\tr^2(\bS^p),
\end{align*}
Now, looking at the terms for $q\ge 2$, using Lemma~\ref{lem:tech2}, and $\frac{1-\left(\frac{2}{3d}\right)^{q-1}}{1-\frac{2}{3d}}\le 2$ for $q\ge 2$ and $d\ge 1$,
\begin{align*}
&2\left(\sum_{ 2\le q< p/2}d^{q-1}+ \sum_{p/2\le q\le p}d^{q-2}\right)\binom{p}{q}3^{q-1}\left(\frac{p}{n}\right)^q\tr^2(\bS^p)
\\ =&
2\sum_{ 2\le q\le p}d^{q-2}\binom{p}{q}3^{q-1}\left(\frac{p}{n}\right)^q\tr^2(\bS^p)
+
2\sum_{ 2\le q< p/2}(d^{q-1}-d^{q-2})\binom{p}{q}3^{q-1}\left(\frac{p}{n}\right)^q\tr^2(\bS^p).
\end{align*}
The first term can be further bounded as,
\begin{align*}
2\sum_{ 2\le q\le p}d^{q-2}\binom{p}{q}3^{q-1}\left(\frac{p}{n}\right)^q\tr^2(\bS^p)
 \le&
\frac{2\tr^2(\bS^p)}{3d^2}\sum_{ 2\le q\le p}\binom{p}{q}\left(\frac{3pd}{n}\right)^q
\\
=&\frac{2\tr^2(\bS^p)}{3d^2}\left[\left(\frac{3pd}{n}+1\right)^p- \frac{3p^2d}{n}-1\right].
\end{align*}
The second term,
\begin{align*}
&2\sum_{ 2\le q < p/2}(d^{q-1}-d^{q-2})\binom{p}{q}3^{q-1}\left(\frac{p}{n}\right)^q\tr^2(\bS^p)
\\
=&
\frac{2(d-1)\tr^2(\bS^p)}{3d^2}\sum_{ 2\le q< p/2}\binom{p}{q}\left(\frac{3pd}{n}\right)^q
\\
\le&
\frac{2(d-1)\tr^2(\bS^p)}{3d^2}d^{p/2}\sum_{ 2\le q< p/2}\binom{p}{q}\left(\frac{3p}{n}\right)^q
\\
\le&
\frac{2(d-1)\tr^2(\bS^p)}{3d^2}d^{p/2}\left(\frac{3p}{n}\right)^{p/2}\sum_{ 2\le q< p/2}\binom{p}{q}
\\
\le&
\frac{2^p(d-1)\tr^2(\bS^p)}{3d^2}d^{p/2}\left(\frac{3p}{n}\right)^{p/2},
\end{align*}
where the last inequality follows immediately from the fact that $\sum_{ 2\le q< p/2}\binom{p}{q}\le 2^{p-1}$ due to the symmetry of the binomial coefficients. 
For $k=0$, we have, 
\begin{align*}
&\binom{n}{p}^{-2}\binom{n}{2p}\binom{2p}{p}\tr^2(\bS^p)=\frac{(n-p)!(n-p)!}{n!(n-2p)!}\tr^2(\bS^p)
=\prod_{k=0}^{p-1}\frac{n-k-p}{n-k}\tr^2(\bS^p)\,.
\end{align*}
$k=1$, we have,
\begin{align*}
&\binom{n}{p}^{-2}\binom{n}{2p-1}\binom{2p-1}{1}\binom{2p-2}{p-1}=\frac{(n-p)!(n-p)!}{n!(n-2p)!}\frac{p^2}{n-2p+1}.
\end{align*}

The upper bound follows. 
\end{proof}
\begin{lem}
\label{lem:tech1}
For non-negative integers $n\ge p\ge q$, we have,
\begin{align*}
\frac{\binom{n}{2p-q}\binom{2p-q}{q}\binom{2p-2q}{p-q}}{\binom{n}{p}}=&\binom{n-p}{p-q}\binom{p}{q}.
\end{align*}
\end{lem}
\begin{proof}
\begin{align*}
\frac{\binom{n}{2p-q}\binom{2p-q}{q}\binom{2p-2q}{p-q}}{\binom{n}{p}}=& \frac{n!(2p-q)!(2p-2q)!(n-p)!p!}{(n-2p+q)!(2p-q)!(2p-2q)!q!(p-q)!(p-q)!n!}\\ =& \frac{(n-p)!p!}{(n-2p+q)!q!(p-q)!(p-q)!}=\binom{n-p}{p-q}\binom{p}{q}.
\end{align*}
\end{proof}
\begin{lem}
\label{lem:tech2}
For non-negative integers $n\ge q\ge p$, we have,
\begin{align*}
\frac{\binom{n-p}{p-q}}{\binom{n}{p}}\le\min \left\{
\begin{array}{l}
\left(\frac{n-p}{n-q}\right)^{p-q}\\
\left(\frac{p}{n}\right)^{q}\,.
\end{array}
\right.
\end{align*}
\end{lem}
\begin{proof}
\begin{align*}
\frac{\binom{n-p}{p-q}}{\binom{n}{p}} = &\frac{(n-p)(n-p-1)\cdots (n-2p-q+1)\times p!}{(p-q)!\times n(n-1)\cdots (n-p+1)}
\\=& \frac{(n-p)(n-p-1)\cdots (n-2p-q+1)\times p(p-1)\cdots (p-q+1)}{n(n-1)\cdots (n-p+1)}
\\=& \frac{(n-p)(n-p-1)\cdots (n-2p-q+1)}{(n-q)(n-q-1)\cdots (n-p+1)}\times \frac{p(p-1)\cdots (p-q+1)}{n(n-1)\cdots (n-q+1)}
\\
\le & \left(\frac{n-p}{n-q}\right)^{p-q}\cdot \left(\frac{p}{n}\right)^{q}
\le\min \left\{
\begin{array}{l}
\left(\frac{n-p}{n-q}\right)^{p-q}\\
\left(\frac{p}{n}\right)^{q}.
\end{array}
\right.
\end{align*}
\end{proof}

\bibliographystyle{plain}

\end{document}